\documentclass[11pt,a4paper]{amsart}
\usepackage{amsmath,amsfonts}
\usepackage{graphicx}
\usepackage{amssymb}
\usepackage{amsthm} 
\usepackage{yfonts}
\usepackage{kpfonts}

\usepackage[usenames,dvipsnames]{pstricks}
\usepackage{epsfig}
\usepackage{rotating}
\usepackage{pst-plot}
\usepackage{pst-eps}
\usepackage{pst-grad}
\usepackage{pstricks-add}
\usepackage{lmodern}
\usepackage{xcolor}
\usepackage{graphicx}
\usepackage[top=2cm,bottom=2cm,left=2cm,right=2cm,a4paper]{geometry}
\usepackage{thm-restate}

\usepackage{etex}
\usepackage[all]{xy}

\usepackage{mathrsfs}
\usepackage{tikz-cd}

\usepackage{cite}

\def\eps{{\varepsilon}}

\newcommand{\qq}{{\mathbb Q}}

\newcommand{\zz}{{\mathbb Z}}

\newcommand{\pp}{{\mathbb P}}

\newcommand{\Pic}{{\rm{Pic}}}
\newcommand{\Kod}{{\rm{Kod}}}

\newcommand{\sh}{\mathscr}

\newcommand{\Mg}{\mathcal{M}_g}
\newcommand{\Mgbar}{\overline{\mathcal{M}}_g}

\newcommand{\Mgnbar}{\overline{\mathcal{M}}_{g,n}}
\newcommand{\Fg}[1]{\mathcal{F}_{#1}}
\newcommand{\Fgn}[2]{\mathcal{F}_{{#1}, {#2}}}
\newcommand{\M}[2]{\mathcal{M}_{{#1}, {#2}}}
\newcommand{\Mbar}[2]{\overline{\mathcal{M}}_{{#1}, {#2}}}

\theoremstyle{plain}
\newtheorem{theorem}{Theorem}[section]
\newtheorem{lemma}[theorem]{Lemma}
\newtheorem{proposition}[theorem]{Proposition}

\theoremstyle{definition}

\title{Two moduli spaces of Calabi-Yau type}
\author[I. Barros]{Ignacio Barros}
\author[S. Mullane]{Scott Mullane}
\address[I. Barros]{
Department of Mathematics\\
Northeastern University\\
360 Huntington Ave.\\
Boston, MA 02115, USA} 
\email{i.barrosreyes@northeastern.edu}

\address[S. Mullane]{
Institut f\"{u}r Mathematik\\
Goethe-Universit\"{a}t Frankfurt\\
Robert-Mayer-Str. 6-8\\
60325 Frankfurt am Main, Germany} 
\email{mullane@math.uni-frankfurt.de}

\begin{document}

\maketitle

\begin{abstract}
We show $\Mbar{10}{10}$ and $\mathcal{F}_{11,9}$ have Kodaira dimension zero. Our method relies on the construction of a number of curves via nodal Lefschetz pencils on blown-up $K3$ surfaces. The construction further yields that any effective divisor in $\overline{\mathcal{M}}_{g}$  with slope $<6+(12-\delta)/(g+1)$ must contain the locus of curves that are the normalization of a $\delta$-nodal curve lying on a $K3$ surface of genus $g+\delta$.
\end{abstract}


\nopagebreak



\section*{Introduction}

It has long been established that for $g\geq2$, the moduli spaces $\Mbar{g}{n}$ are all of general type except for finitely many pairs $(g,n)$ occuring in relatively low genus. Hence though the description of the map
$$(g,n)\mapsto \rm{Kod}\left(\Mbar{g}{n}\right)$$ 
is far from being complete, it is unknown only for relatively small $g$ and $n$. Since the Severi's conjecture \cite{S} was disproven by Harris and Mumford \cite{HM}, the complete description of the mentioned map has become a major task in the study of moduli spaces of curves. See \cite{Se,CR1,CR2,EH1,EH2,L,BV,FP,F}. 

For a fixed $g$, when $n$ increases the transition from Kodaira dimension $-\infty$ to general type is rather sudden due to the fact that the forgetful map 
$$\pi:\Mbar{g}{n+1}\to\Mbar{g}{n}$$
has fibers of general type and the dualizing sheaf $\omega_{\pi}$, though not big, has some positivity properties discussed briefly in \S \ref{Preliminaries}.

In almost all the known cases when $\Mbar{g}{n}$ is not of general type it has some rational attributes, namely, it is at least uniruled. The only exception of this is when $(g,n)=(11,11)$, where ${\rm{Kod}}\left(\Mbar{11}{11}\right)=19$, cf. \cite[Thm. 0.5]{FV1}. Until now, this was the only example of intermediate Kodaira type for $g\geq2$. 
 
In this paper we continue the search initiated by Farkas and Verra \cite{FV1} for moduli spaces $\Mbar{g}{n}$ of intermediate type, with the hope of shedding some light on the mysterious transition from uniruled to general type when $g$ or $n$ increase. The main result of this paper concerns the genus $g=10$ case.

\begin{theorem}
\label{theorem1}
The moduli space $\Mbar{10}{10}$ has Kodaira dimension zero.
\end{theorem}

Much less is known when it comes to moduli spaces of $K3$ surfaces. Following Mukai's notation, let $\mathcal{F}_g$ be the moduli space of primitively polarized K3 surfaces of genus $g$. For $K3$ surfaces the same principle holds, meaning that all but finitely many $\mathcal{F}_g$'s are of general type, cf. \cite[Thm. 1]{GHS}. Mukai in a series of papers \cite{M1,M2,M3,M4,M5} established structure theorems of $\mathcal{F}_g$, for small $g$. From Mukai's results it follows that $\mathcal{F}_g$ is unirational for $g\leq 13$ and $g=16, 18, 20$.

Let $\mathcal{F}_{g,n}$ be the moduli space of $n$-pointed polarized $K3$ surfaces of genus $g$, consisting of tuples $(S,H,x_1,\ldots,x_n)$, where $(S,H)$ is a primitively polarized $K3$ surface of genus $g$ and $x_1,\ldots, x_n\in S$ are $n$ ordered points on $S$. Notice that when $n\geq 1$, the moduli space $\Fgn{g}{n}$ is never of general type. Indeed, the forgetful map
$$\Fgn{g}{n}\to\Fg{g}$$
is a morphism fibered in Calabi-Yau varieties and hence by Iitaka's addition formula 
$$\Kod\left(\Fgn{g}{n}\right)\leq \dim\Fg{g}=19.$$
However, by \cite{K} 
$$\Kod(\Fgn{g}{n+1})\geq \Kod(\Fgn{g}{n}).$$ 
Hence we observe that the complexity of $\Fgn{g}{n}$ cannot decrease when the number of markings increases. Further, when $\Fg{g}$ is of general type, $\Fgn{g}{n}$ has maximal Kodaira dimension equal to $19$. It remains unknown if all but finitely many $\Fgn{g}{n}$'s have Kodaira dimension $19$. However, $\Fgn{g}{n}$ has maximal Kodaira dimension for $g\geq62$ and some other values $g\geq 47$, cf. \cite[Thm. 1]{GHS}. It has also maximal Kodaira dimension for $(g,n)=(11,11)$ and negative Kodaira dimension for $(g,n)=(14,1)$, $g\leq 10$ and $n\leq g+1$, or $g=11$ and $n\leq 7$, cf. \cite[Thm. 1.1, Thm. 5.1]{FV2} and \cite[Thm. 1]{B}. The second of our findings is the following:

\begin{theorem}
\label{theorem2}
The moduli space $\mathcal{F}_{11,9}$ has Kodaira dimension zero.
\end{theorem}

Finally, using the relation between symmetric quotients of $\Mbar{10}{10}$ and $\Fgn{11}{n}$ established in \cite[Thm. 0.4]{B}, it follows from Theorem \ref{theorem1}, Theorem \ref{theorem2} and
$$\Kod(\Mbar{10}{10}\big/S_{10})=0 $$ 
\cite[Thm 0.1]{FV1} and \cite[Thm 1.2]{BFV}, that all intermediate quotients are of Calabi-Yau type.

\begin{theorem}
\label{theorem3}
For any subgroup $G\leq S_{10}$, the moduli space $\Mbar{10}{10}\big/G$ has Kodaira dimension zero. Similarly, for any subgroup $H\leq S_8$ that acts by fixing one of the markings, the moduli space ${\mathcal{F}}_{11,9}\big/H$ has Kodaira dimension zero.
\end{theorem}

The strategy broadly follows \cite{FV1}. The canonical class $K_{\Mbar{10}{10}}$ can be expressed as an effective linear combination of Logan's divisor (see \S 2), the pullback of the closure of the $K3$ locus in $\overline{\mathcal{M}}_{10}$ (class computed in \cite{FP}), and boundary divisors. Each of these irreducible components is rigid and we show that the canonical divisor is rigid through the construction of covering curves of each component with negative intersection with the canonical divisor. The covering curves are constructed from Lefschetz pencils on blown-up $K3$ surfaces. 
Theorem \ref{theorem2} follows from the birational map
$$\Mbar{10}{10}\big/S_2\dashrightarrow\mathcal{F}_{11,9}$$ 
constructed in \cite{B}, where $S_2$ permutes the first two marked points.

We complete the paper with two further applications of the curves constructed above. We first obtain an extremal nef curve in $\overline{\mathcal{M}}_{10}$ which shows that $\mathbb{Q}^+\langle \overline{\mathcal{K}},\delta_1,\dots \delta_5\rangle$ forms a face of the pseudo-effective cone of divisors in $\overline{\mathcal{M}}_{10}$, where $ \overline{\mathcal{K}}$ is the closure of the $K3$ locus. 

Finally, the negative intersection of an effective divisor with the irreducible curves constructed in Section~\ref{Lefschetz} implies inclusion. Hence we obtain the following.

 \begin{restatable}{theorem}{theoremfour}
 \label{theorem4}
Any effective divisor $D$ in $\overline{\mathcal{M}}_{g}$ that is the closure of an effective divisor on $\Mg$ with slope $s(D)<6+(12-\delta)/(g+1)$ must contain the locus of curves that are the normalization of a $\delta$-nodal curve lying on a $K3$ surface of genus $g+\delta$.
\end{restatable}

 In Section~\ref{Preliminaries} we recall the definition and class of Logan's divisor and the $K3$ divisor in $\overline{\mathcal{M}}_{10}$ in terms of the standard generators of the Picard group of $\Mbar{g}{n}$. In Section~\ref{Lefschetz} we construct nodal Lefschetz pencils on blown-up $K3$ surfaces and compute the numerical invariants. In Section~\ref{KodM1010} we prove Theorem \ref{theorem1} and in Section~\ref{SymQuot} we prove Theorem \ref{theorem2} and  show Theorem \ref{theorem3} follows. In Section~\ref{FurtherApplications} we discuss the further applications of the curves constructed in any genus in Section~\ref{Lefschetz}, including the construction of a new extremal nef curve in $\overline{\mathcal{M}}_{10}$ and the associated face of the psuedo-effective cone and we prove Theorem~\ref{theorem4}.

\section*{Acknowledgements}
We would like to thank Gavril Farkas, Rahul Pandharipande, Ana-Maria Castravet and Paul Hacking for several stimulating discussions regarding this paper. We would
also like to thank the anonymous referees for their useful comments
and corrections. The second author was supported by the Alexander von Humboldt Foundation during the preparation of this article.

\section{Preliminaries}
\label{Preliminaries}
Let $g\geq 2$ be the genus. The relative dualizing sheaf $\omega_\pi$ of the universal curve 
$$\pi:\Mbar{g}{1}\to\Mgbar$$ 
is big, cf. \cite[\S 3]{CHM}, and the relative dualizing sheaf of $p:\Mbar{g}{n+1}\to\Mbar{g}{n}$ is the pullback of $\omega_\pi$ by the forgetful map $q:\Mbar{g}{n+1}\to\Mbar{g}{1}$. Since the relative dualizing sheaf $\omega_p$ and the sheaf of relative one forms $\Omega_p$ differ in codimension two, 
$$K_{\Mbar{g}{n+1}}\cong p^*K_{\Mbar{g}{n}}\otimes q^*\omega_\pi.$$
Hence when a multiple of $K_{\Mbar{g}{n}}$ is effective, that is, $\Kod(\Mbar{g}{n})\geq 0$, we have
$$\Kod\left(\Mbar{g}{n+1}\right)\geq 3g-2.$$
This explains the sudden growth of $\Kod(\Mbar{g}{n})$ when $n$ increases and the salience of moduli spaces $\Mbar{g}{n}$ of intermediate type.\\ 

As discussed in the introduction, we first study the $S_2$-quotient of $\Mbar{10}{10}$. It was established in \cite[Thm. 0.4]{B} that there is a birational isomorphism
$$\Mbar{10}{10}\big/S_2\dashrightarrow\mathcal{F}_{11,9}$$
sending a general point in the source $[C,p_1+p_2,p_3,\ldots,p_{10}]$ to the unique pointed polarized K3 surface $(S,H,x_1,p_3,\ldots,p_{10})$ such that, after identifying the two first unordered points, the genus eleven nodal curve $C\big/p_1\sim p_2$ lies in the linear system $|H|$ and passes through $p_3,\ldots,p_{10}\in S$. The forgetful map 
$$\mathcal{F}_{11,9}\to\mathcal{F}_{11}$$
is fibered in Calabi-Yau varieties and hence by the Easy Addition Formula,
$$\Kod\left(\Mbar{10}{10}\big/S_2\right)\leq 19.$$
On the other hand, by \cite[Thm. 0.1]{FV1} and \cite[Thm. 1.2]{BFV}  
$$\Kod\left(\Mbar{10}{10}\big/S_{10}\right)=0.$$
It follows that all intermediate symmetric quotients are of intermediate type. It is worth mentioning that $\Mbar{g}{n}\big/S_n$ has finite quotient singularities and they do not impose adjunction conditions, i.e., pluricanonical forms on the regular locus extend to any desingularization. See \cite[Thm. 1.1]{FV1} and \cite[Thm. 1.4]{BFV}.   \\

Recall that $\Pic_{\mathbb{Q}}(\Mbar{g}{n})$ is generated by $\lambda$, the first Chern class of the Hodge bundle, $\psi_i$ the first Chern class of the cotangent bundle on $\Mbar{g}{n}$ associated with the $i$th marking for $i=1,\dots,n$ and the classes of the irreducible components of the boundary. We denote by $\delta_{\rm{irr}}$ the class of the locus of curves with a non-separating node and $\delta_{i:S}$ for $0\leq i\leq g$, $S\subset\{1,\dots,n\}$ the class of the locus of curves with a separating node that separates the curve such that one component has genus $i$ and contains precisely the markings of $S$. Hence $\delta_{i:S}=\delta_{g-i:S^c}$ and we require $|S|\geq 2$ for $i=0$ and $|S|\leq n-2$ for $i=g$. For $g\geq 3$ these divisors freely generate $\Pic_{\mathbb{Q}}(\Mbar{g}{n})$.

The canonical class of $\Mgnbar$ can be computed using Grothendieck-Riemann-Roch and it is given by  the formula 
\begin{equation}
\label{canMgn}
K_{\Mbar{g}{n}}=13\lambda+\sum_{i=1}^n\psi_i -2\delta_{\rm{irr}}-2\sum_{\substack{S\subset\{1,\ldots,n\}\\ i\geq 0}}\delta_{i:S}-\delta_{1:\varnothing}.
\end{equation}

Following \cite{FV1}, there are two divisors that play a central role in showing that $K_{\Mbar{10}{10}\big/S_2}$ is effective and rigid. The first is the divisor $D_g\subset\M{g}{g}$ consisting of pointed curves $[C,p_1,\ldots,p_g]$ such that $p_1+\ldots+p_g\in\rm{Div}(C)$ moves on a pencil, that is, 
$$D_g:=\left\{\left[C,p_1,\ldots,p_g\right]\in\M{g}{g}\mid h^0(C,\mathcal{O}_C(p_1+\ldots+p_g))\geq2\right\}.$$ 
The class of the closure $\overline{D}_g\subset\Mbar{g}{g}$ was computed in \cite[Thm. 5.4]{L}: 
\begin{equation}
\label{Logan}
\overline{D}_g=-\lambda+\psi-\sum_{\substack{S\subset\{1,\ldots,g\}\\ 0\leq i\leq \lfloor g/2\rfloor}}\binom{||S|-i|+1}{2}\delta_{i:S}\in{\rm{Pic}}_{\qq}\left(\Mbar{g}{g}\right).
\end{equation}
The second divisor that plays a key role is the closure of the locus of curves lying on a K3 surface;
$$\mathcal{K}:=\left\{[C]\in\mathcal{M}_{10}\mid \hbox{ there exists } (S,H)\in\mathcal{F}_{10}\hbox{ with }C\in|H|\right\}.$$
It was showed in \cite[Thm. 1.6]{FP} that $\overline{\mathcal{K}}\subset\overline{\mathcal{M}}_{10}$ is an irreducible divisor of class
\begin{equation}
\label{K3divisor}
\overline{\mathcal{K}}=7\lambda-\delta_{\rm{irr}}-5\delta_1-9\delta_2-12\delta_3-14\delta_4-B\delta_5\in{\rm{Pic}}_{\qq}\left(\overline{\mathcal{M}}_{10}\right),
\end{equation}  
where $B\geq 6$. If we call $\pi:\Mbar{10}{10}\to\overline{\mathcal{M}}_{10}$, from (\ref{canMgn}), (\ref{Logan}), and (\ref{K3divisor}), we have the following effective decomposition of the canonical class;
\begin{equation}
\label{canM1010}
K_{\Mbar{10}{10}}=\overline{D}_{10}+2\pi^*\overline{\mathcal{K}}+\sum_{i,S}c_{i:S}\delta_{i:S},
\end{equation}
where $c_{i:S}>0$ for all $i$ and $S$.\\

In the next section we construct several irreducible curves in $\Mbar{g}{n}$ that for $(g,n)=(10,10)$ cover all irreducible components of the effective decomposition (\ref{canM1010}). Negative intersection of such a curve with an effective divisor shows the divisor lies in the base locus of the linear system. Through the repeated application of such curves we show the effective divisors $nK_{\Mbar{10}{10}}$ to be rigid for all $n\geq1$.

\section{Numerical invariants for a generalized Lefschetz pencil}\label{Lefschetz}

Let $(S,H, x_1,\ldots,x_\delta,y_1,\ldots,y_l)\in \mathcal{F}_{g,\delta+l}$ be a polarized K3 of genus $g$ with $\delta+l$ marked points and
$$\eps:\tilde{S}\to S$$
the blow-up of S at $x_1,\ldots,x_\delta,y_1,\ldots,y_l$. We fix 
$$E=E_1+\ldots+E_\delta,\hspace{0.2cm}\hbox{and}\hspace{0.2cm}F=F_1+\ldots+F_l,$$
the sum of exceptional divisors over the points $x_i$ and $y_j$ respectively. Assume
$$h^0\left(\tilde{S},\mathcal{O}_S\left(\eps^*H-2E-F\right)\right)\geq2$$
and let 
$$\pp^1\subset|\eps^*H-2E-F|$$
be a general pencil in the linear system with smooth general element. The pencil $\pp^1$ induces a family of curves of genus $g-\delta$ in $\tilde{S}$. Notice that the base locus of the pencil consists of 
$$(\eps^*H-2E-F)^2=2(g-\delta)-2-2\delta-l$$
points. We resolve the family by blowing up at the base locus of the pencil
\begin{displaymath}
\begin{tikzcd}
{\sh{U}}\ar[r]\ar[dr, "\phi"']&\tilde{S}\ar[dashed, d]\\
&\pp^1.
\end{tikzcd}
\end{displaymath}
Abusing notation we call $E_i\subset\sh{U}$ and $F_i\subset\sh{U}$ the proper transforms of the exceptional divisors in $\tilde{S}$. We want to construct a curve on 
$\overline{\mathcal{M}}_{g-\delta,2\delta+l}$. Notice that 
$$\phi\mid_{F_i}:F_i\to\pp^1$$ 
is an isomorphism and 
$$\phi\mid_{E_i}:E_i\to\pp^1$$
is a degree two morphism ramified over two points. The family $\phi$ comes with $l$ sections $\delta$ many $2$-sections. For any $2$-section, the induced map $E_i\to\pp^1$ is ramified at two points and choosing the pointed $K3$ and the pencil general, we can assume that the ramification points are all different. In order to separate the sections we perform the degree $2$ base change  

\begin{displaymath}
\begin{tikzcd}
\sh{U}\times E_1\ar[d, "p_1"]\ar[r, "f_1"]&\sh{U}\ar[d ,"\phi"]\\
E_1\ar[r]&\pp^1
\end{tikzcd}
\end{displaymath}
where the map at the bottom is $\phi$ restricted to $E_1$. Notice that $f_1$ is a two-to-one cover ramified at two distinct fibers $C_{11}$ and $C_{12}$ of $\sh{U}\times E_1\to E_1$, and $E_1\times E_1$ consist of two copies of $E_1$, one given by the diagonal and the one by the locus $(x,\bar{x})$, where $\bar{x}$ is the conjugate of $x$ for the map $E_1\to \pp^1$. Thus,
$$f_1^*E_1=E_{11}+E_{12},$$ 
and the two components intersect in exactly two points. Notice that the induced map 
$$p_1\mid_{E_1\times E_2}:E_1\times E_2\to E_1$$ 
is again a $2$-to-$1$ cover ramified over two points. We continue this process $\delta$-times to obtain a family with $2\delta+\ell$ sections; $E_{11}, E_{12}, \ldots,E_{\delta1},E_{\delta2}, F_1,\ldots,F_{\ell}$. For $i=1,\ldots,\delta$, the section $E_{i1}$ intersects $E_{i2}$ in two points. In order to separate them we blow up the corresponding $2\delta$ points, $\tau:\sh{X}\to\sh{U}_{\delta}$. The resulting family 
$$\pi:\sh{X}\to \Gamma(g,\delta,l), \hspace{0.2cm} s_i:\Gamma\to\sh{X}, i=1,\ldots,2\delta+l$$
is the blow-up of a $2^\delta$-to-$1$ cover of $\sh{U}$ and it has stable fibers, inducing a family on $\Mbar{g-\delta}{2\delta+l}$. We will assume that the sections are ordered so that the first $2\delta$ sections correspond to the proper transforms $E_{11},E_{12},\ldots,E_{\delta1},E_{\delta2}$ and the last $\ell$ sections correspond to the proper transforms $F_i$'s.

\begin{displaymath}
\begin{tikzcd}
{\sh{X}}\ar[r,"\tau"]\ar[dr, "\pi"']&\sh{U}_\delta\ar[r, "f"]\ar[d]&\sh{U}\ar[d, "\phi"]\\
&\Gamma:=E_1\times\ldots\times E_\delta\ar[r]&\pp^1.
\end{tikzcd}
\end{displaymath}

\begin{proposition}
\label{intersections}
We have the following intersection numbers of the curve $\Gamma(g,\delta,l)$ with the standard divisors on $\Mbar{g-\delta}{2\delta+\ell}$:
$$
\begin{array}{lcl}
\Gamma\cdot\lambda=2^{\delta}(g-\delta+1),&&\Gamma\cdot\psi_i=2^{\delta}, \hbox{for }i=2\delta+1,\ldots,2\delta+\ell,\\
\Gamma\cdot\delta_{\rm{irr}}=2^{\delta}(18+6g-7\delta),&&\Gamma\cdot \delta_{0:\{1,2\}}=\Gamma\cdot\delta_{0:\{3,4\}}=\ldots=\Gamma\cdot\delta_{0:\{2\delta-1,2\delta\}}=2,\\
\Gamma\cdot\psi_{i}=2^{\delta-1}+4, \hbox{for }i=1,\ldots,2\delta,&&\Gamma\cdot\delta_{i:S}=0\text{ otherwise.} 
\end{array}
$$
\end{proposition}

\begin{proof}
The map $f:\sh{U}_{\delta}\to\sh{U}$ is ramified at $(2^{\delta+1}-2)$ fibers. Thus,  
\begin{equation}
\label{intersections.eq1}
K_{\sh{X}}=(\tau\circ f)^*K_{\sh{U}}+(2^{\delta+1}-2)C+G,
\end{equation}
where $C$ is the class of a fiber of $\pi$ and $G=G_1+\ldots+G_{2\delta}$ is the exceptional divisor of the blow-up $\tau:\sh{X}\to\sh{U}_{\delta}$. Using (\ref{intersections.eq1}) observe
\begin{equation}
\begin{array}{c}
c_1(\sh{X})^2=(6\cdot2^{\delta}-8)g+(6-5\cdot2^{\delta})\delta+8-6\cdot2^{\delta},\\
c_2(\sh{X})=(6\cdot2^\delta-4)g+(6-7\cdot2^\delta)\delta+18\cdot2^\delta+4
\end{array}
\end{equation}
To compute $\Gamma\cdot\lambda$ we use Noether's formula together with the formula 
$$\deg(\lambda_\pi)=\chi(\sh{X},\mathcal{O}_{\sh{X}})-\chi(\Gamma,\mathcal{O}_\Gamma)\cdot\chi(C,\mathcal{O}_C),$$
where $C$ is a general fiber of $\pi$, cf. \cite[Prop. 3.1]{CU}. To compute $\Gamma\cdot\delta_{\rm{irr}}$ we compare the topological Euler characteristic $\chi_{\rm{top}}(\pp^1)\cdot\chi_{\rm{top}}(C)$ of the product $\pp^1\times C$ with the topological Euler characteristic of the surface $\sh{X}$. To compute the intersection with the $\psi$ classes, observe $(\tau^*(E_{i1}+E_{i2}))^2=-2^{\delta}$ and $E_{i1}\cdot E_{i2}=2$, giving the self intersection of the proper transforms on $\sh{X}$. The argument is the same for $F_i$. The remaining zero intersections are clear from the construction. 
\end{proof}

\section{The Kodaira dimension of $\Mbar{10}{10}$ is zero}\label{KodM1010}

It was proved in \cite[Thm. 0.4]{B} that for $\delta$ and $\ell$ in certain range, some symmetric quotients of $\Mbar{11-\delta}{2\delta+\ell}$ are birational to $\Fgn{11}{\delta+\ell}$. In particular for $\delta=1$ and $\ell=8$, the result is the following:
\begin{theorem}[Thm. 0.4 in \cite{B}]
\label{F11,9}
There exists a birational map
$$\Mbar{10}{10}\big/S_2\dashrightarrow\Fgn{11}{9}$$
sendings a general pointed curve $[C,x_1+x_2,x_3,\ldots,x_{10}]\in\Mbar{10}{10}\big/S_2$ to the unique pointed polarized K3 surface $(S,H,y,x_3,\ldots,x_{10})\in \Fgn{11}{9}$, where the nodal curve $C\big/x_1\sim x_2$ obtained by identifying $x_1=x_2=y$ lies in the linear system $|H|$. 

\end{theorem}

As a corollary we have the following lemma.

\begin{lemma}
\label{lemma1}
The curve $\Gamma(11,1,8)$ is a covering curve for $\overline{D}_{10}$. Moreover,
$$\Gamma(11,1,8)\cdot \overline{D}_{10}=-2\hbox{ and }\Gamma(11,1,8)\cdot K_{\Mbar{10}{10}}=0.$$
\end{lemma}

\begin{proof}
Let $[C,x_1,\ldots,x_{10}]\in\Mbar{10}{10}$ be a general point on $\overline{D}_{10}$. Since the curve $[C,x_1+x_2]\in\Mbar{10}{2}\big/S_2$ is general, there exists a blown-up K3 surface, $\eps:\tilde{S}\to S$, with exceptional divisor $E$, such that the curve $C$ is in the linear system $|\eps^*H-2E|$, with $(S,H)$ in $\mathcal{F}_{11}$. On the other hand, since $h^0(C,\mathcal{O}_{C}(x_1+\ldots+x_{10}))=2$, by the exact sequence
$$0\to \mathcal{O}_{\tilde{S}}\to\mathcal{I}_{x_1+\ldots+x_{10}/\tilde{S}}(C)\to \omega_{C}(-x_1-\ldots-x_{10})\to0,$$
there exists a pencil $\pp^1\cong\pp H^0(\tilde{S},\mathcal{I}_{x_1+\ldots+x_{10}/\tilde{S}}(C))\subset|\eps^*H-2E|$, such that the induced curve on $\Mbar{10}{10}$ is given by $\Gamma(11,1,8)$. The intersection numbers follow from Proposition \ref{intersections}.
\end{proof}

Let 
$$\sigma:\Mbar{10}{10}\longrightarrow\Mbar{10}{10}$$ 
permute the first and last markings, 
$$\varphi:\Mbar{10}{10}\longrightarrow\Mbar{10}{9}$$
forget the second marked point and 
$$\theta:\Mbar{10}{\{3,4,\ldots,10\}\cup\{r\}}\to\Mbar{10}{10}$$
be the boundary map that associates to any $\{r\}\cup\{3,4,\ldots,10\}$-pointed stable curve of genus $10$, the $10$-pointed curve obtained by gluing to it a $\{1,2,s\}$-pointed rational tail by identifying $r$ with $s$.

\begin{lemma}\label{lemma2}
The curve $\underline{\Theta}:=\theta_*\varphi_*\sigma_*\Gamma(11,1,8)$ is a covering curve for $\delta_{0:\{1,2\}}$ and
$$\underline{\Theta}\cdot\lambda=22,\hspace{1cm} \underline{\Theta}\cdot\delta_{\rm{irr}}=154,\hspace{1cm}\underline{\Theta}\cdot\delta_{0:\{1,2\}}=-2, \hspace{1cm} \underline{\Theta}\cdot\psi_{10}=3,$$
$$ \underline{\Theta}\cdot\psi_{i}=2 \text{ for $i=3,\dots,9$.}$$
\end{lemma}

\begin{proof}
This follows by Proposition \ref{intersections} and standard pull-back formulas for $\lambda, \psi_i$ and the boundary divisors, cf. \cite[Ch. XVII, Lemma 4.38]{ACG}.
\end{proof}

We fix the effective divisor
$$A:=K_{\Mbar{10}{10}}-\overline{D}_{10}-\sum_{|S|\geq 2}c_{0:S}\delta_{0:S}=2\pi^*\overline{K}+\sum_{\substack{1\leq i\leq 5\\ S}}c_{i:S}\delta_{i:S},$$
see (\ref{canM1010}).

\begin{proposition}
\label{mainProp}
The Kodaira dimension of $\Mbar{10}{10}$ equals the Iitaka dimension of $A$.
\end{proposition}

\begin{proof}
By Lemmas \ref{lemma1} and \ref{lemma2}, $\Gamma(11,1,8)$ and $\underline{\Theta}$ form a covering family of curves for Logan's divisor $\overline{D}_{10}$ and $\delta_{0:\{1,2\}}$ respectively. Observe
$$ \underline{\Theta}\cdot K_{\Mbar{10}{10}}=-1,\hspace{1cm} \underline{\Theta}\cdot\delta_{0:\{1,2\}}=-2,\hspace{1cm} \underline{\Theta}\cdot \overline{D}_{10}=1$$
$$ \Gamma(11,1,8)\cdot K_{\Mbar{10}{10}}=0,\hspace{1cm} \Gamma(11,1,8)\cdot\delta_{0:\{1,2\}}=2,\hspace{1cm} \Gamma(11,1,8)\cdot \overline{D}_{10}=-2.$$
Hence we obtain for $m\in\mathbb{N}$,
\begin{equation} \underline{\Theta}\cdot \left(K_{\Mbar{10}{10}} -\frac{2^m-1}{2^m}\delta_{0:\{1,2\}}-\frac{2^m-1}{2^m}\overline{D}_{10}  \right)=\frac{-1}{2^m}\end{equation}
and
\begin{equation} \Gamma(11,1,8)\cdot \left(K_{\Mbar{10}{10}} -\frac{2^{m+1}-1}{2^{m+1}}\delta_{0:\{1,2\}}-\frac{2^m-1}{2^m}\overline{D}_{10}  \right)=\frac{-1}{2^m}.\end{equation}
Provided the divisor on the LHS is effective, we obtain
\begin{equation}
\label{takedelta}
\resizebox{0.9\hsize}{!}{$
\left|n\left(K_{\Mbar{10}{10}} -{\frac{2^m-1}{2^m}}\delta_{0:\{1,2\}}-\frac{2^m-1}{2^m}\overline{D}_{10}  \right)\right|=\left|n\left(K_{\Mbar{10}{10}} -\frac{2^{m+1}-1}{2^{m+1}}\delta_{0:\{1,2\}}-\frac{2^m-1}{2^m}\overline{D}_{10}  \right)  \right|+\frac{n}{2^{m+1}}\delta_{0:\{1,2\}}
$}
\end{equation}
and
\begin{equation}\label{takelogan} 
\resizebox{0.9\hsize}{!}{$
\left|n\left(K_{\Mbar{10}{10}} -\frac{2^{m+1}-1}{2^{m+1}}\delta_{0:\{1,2\}}-\frac{2^m-1}{2^m}\overline{D}_{10}  \right)  \right|=
\left|n\left(K_{\Mbar{10}{10}} -{\frac{2^{m+1}-1}{2^{m+1}}}\delta_{0:\{1,2\}}-\frac{2^{m+1}-1}{2^{m+1}}\overline{D}_{10}  \right)\right|+\frac{n}{2^{m+1}}\overline{D}_{10}
$}
\end{equation}
Hence starting at $m=0$ where the LHS of Equation~(\ref{takedelta}) is $|nK_{\Mbar{10}{10}}|$ and hence effective, iteratively apply Equation (\ref{takedelta}) then (\ref{takelogan}) for increasing values of $m$ to obtain in the limit
\begin{equation*}
\left|n\left(K_{\Mbar{10}{10}} \right)  \right|=
\left|n\left(K_{\Mbar{10}{10}} -\delta_{0:\{1,2\}}-\overline{D}_{10}  \right)\right|+n\delta_{0:\{1,2\}}+n\overline{D}_{10}
\end{equation*}
Observe that by symmetry or permuting the markings of the curve $\underline{\Theta}$ we obtain
$$\left|n\left(K_{\Mbar{10}{10}} \right)  \right|=
\left|n\left(K_{\Mbar{10}{10}} -\sum_{|S|=2}\delta_{0:S}-\overline{D}_{10}  \right)\right|+n\sum_{|S|=2}\delta_{0:S}+n\overline{D}_{10}.    $$
We now repeat this process for the other boundary components of the form $\delta_{0:S}$. Fix 
$$T=\{1,2,\dots,t\}\hbox{ for }t=2,\dots 10.$$ 
Let 
$$\sigma:\Mbar{10}{10}\longrightarrow\Mbar{10}{10}$$ 
permute the first and tenth markings, 
$$\varphi:\Mbar{10}{10}\longrightarrow\Mbar{10}{11-s}$$
forget the second to $t$th marked points and 
$$\theta:\Mbar{10}{\{t+1,\ldots,10\}\cup\{r\}}\to\Mbar{10}{10}$$
be the boundary map that associates to any $\{r\}\cup\{t+1,\ldots,10\}$-pointed stable curve of genus $10$, the $10$-pointed curve obtained by gluing to it a $\{1,\dots,t,s\}$-pointed rational tail by identifying $r$ with $s$. 

The curve $\underline{\Theta}_T:=\theta_*\varphi_*\sigma_*\Gamma(11,1,8)$ is a covering curve for $\delta_{0:\{1,\dots,t\}}$ with $\underline{\Theta}_T\cdot \delta_{0:\{1,\dots,t\}}=-2$. Further, observe
$$\underline{\Theta}_T\cdot \left(K_{\Mbar{10}{10}}-\sum_{|S|\leq t-1}\delta_{0:T}-\overline{D}_{10}\right)=-2\left(\frac{(t+1)t}{2}-2\right)$$
where $-(t+1)t/2$ is the coefficient of $\delta_{0:T}$ in $\overline{D}_{10}$. Hence by symmetry we inductively obtain
$$
\left|n\left(K_{\Mbar{10}{10}} \right)  \right|=
\left|nA\right|+n\sum_{T\subset [10]}\left(\frac{(t+1)t}{2}-2\right)\delta_{0:T}+n\overline{D}_{10}.$$
\end{proof}

Now we prove our main theorem.

\begin{proof}[Proof of Theorem \ref{theorem1}]
By Proposition \ref{mainProp}, we have reduced our problem to compute $\kappa(A)$. For $i=1,\ldots,5$ we fix $d_i=\max\left\{c_{i:S}\mid S\subset\{1,\ldots,10\}\right\}$ and 
$$B=2\overline{\mathcal{K}}+d_1\delta_1+\ldots+d_5\delta_5\in\rm{Pic}\left(\overline{\mathcal{M}}_{10}\right).$$
The divisor $B$ is supported over the K3 locus and higher boundary divisors. By Proposition \ref{intersections}, the K3 locus is covered by curves $\Gamma(10,0,0)$ with
$$\Gamma(10,0,0)\cdot \overline{\mathcal{K}}=-1\hspace{0.3cm}\text{ and  }\hspace{0.3cm}\Gamma(10,0,0)\cdot B=-2$$
see also \cite[Lemma 2.1]{FP}. The divisor $d_1\delta_1+...+d_5\delta_5$ is well-known to be rigid for non-negative $d_i$. For example, $\delta_i$ is covered by curves 
$\xi_*\Gamma(10-i,0,1)$, where 
$$\xi:\Mbar{10-i}{1}\to\overline{\mathcal{M}}_{10}$$
is the map that attaches to $[C,x]\in\Mbar{10-i}{1}$ a fixed general curve of genus $i$ by identifying $x\in C$ with a general point on the fixed curve. Hence
$$\xi_*\Gamma(10-i,0,1)\cdot\delta_i=-1\hspace{0.3cm}\text{ and  }\hspace{0.3cm}\xi_*\Gamma(10-i,0,1)\cdot(d_1\delta_1+...+d_5\delta_5)=-d_i.$$
Thus,
$$\kappa(A)\leq\kappa(\pi^*B)=0.$$

\end{proof}

\section{The Kodaira dimension of $\Fgn{11}{n}$ and symmetric quotients of $\Mbar{10}{10}$}\label{SymQuot}

As mentioned in the introduction, much less in known about the Kodaira dimension of $\Fgn{g}{n}$. We know it ranges between $-\infty$ and $19$, and that as soon as ${\rm{Kod}}\left(\Fgn{g}{n}\right)$ is maximal, ${\rm{Kod}}\left(\Fgn{g}{m}\right)=19$ for all $m\geq n$.\\

Consider the universal K3 surface $\pi:\Fgn{g}{1}\to{\mathcal{F}}_g$. Recall that (e.g. \cite[Section 0]{PY}) 
$$c_1(\omega_{\pi})=\pi^*\lambda,$$
where $\lambda=c_1({\mathbb{E}})$ is the first Chern class of the Hodge line bundle, whose fiber over $(S,H)\in {\mathcal{F}}_g$ is $H^0(S,K_S)$. Moreover, $\lambda$ is the restriction of an ample class in the Baily-Borel compactification $\overline{\mathcal{F}}_g^*$ of ${\mathcal{F}}_g$. Observe that 
\begin{equation}
\label{Fgn}
K_{\Fgn{g}{n+1}}=p^*K_{\Fgn{g}{n}}\otimes q^*\omega_\pi,
\end{equation}
where $p:\Fgn{g}{n+1}\to\Fgn{g}{n}$ forgets the last marked point and $q:\Fgn{g}{n+1}\to \Fgn{g}{1}$ forgets all but the last marked point.

In analogy with $\Mbar{g}{n}$ where for fixed $g$ the change from uniruled to general type occurs suddenly as $n$ grows, cf. \S1, for K3 surfaces from (\ref{Fgn}) it is reasonable to expect that the transition is either immediate or happens in one step, i.e., for fixed $g$ there is at most one $n$ such that $\Fgn{g}{n}$ is of intermediate Kodaira type. There are several compactifications of $\mathcal{F}_g$ with singularities that do not impose adjunction conditions, i.e, where pluricanonical forms extend to a desingularization (for example canonical singularities for toroidal compactifications, see \cite{G}) and the Hodge class $\lambda$ extends to a big class. The difficulty is that none of those come with a universal family. The question is:

Is there a compactification of the universal K3 surface $\overline{\pi}:\overline{\mathcal{F}}_{g,1}\to\overline{\mathcal{F}}_g$, with canonical singularities, where $\kappa(\omega_{\overline{\pi}})=19$? 

If this question has a positive answer, the Kodaira classification of $\Fgn{g}{n}$ for fixed $g$ is complete if a value of $n$ such that $\Fgn{g}{n}$ is of intermediate type is found. 

Recall that in genus $g=11$, Mukai \cite{M6} proved that a general curve $C$ of genus $11$ has a unique $K3$ extension $S$ up to isomorphism, moreover $\Pic(S)=\zz\cdot\left[C\right]$. This construction induces a rational map
$$\begin{array}{rcl}
\psi:\M{11}{n}&\dashrightarrow&\Fgn{11}{n}\\
\left[C,x_1,\ldots,x_n\right]&\mapsto&\left(S,{\mathcal{O}}_S(C),x_1,\ldots,x_n\right),
\end{array}$$
which is dominant for $n\leq 11$ and birational for $n=11$.\\

From this follows the unirationality of $\Fg{11}$. On the other hand in \cite[Thm. 5.1]{FV2} it is proved that $\Fgn{11}{11}$ has positive Kodaira dimension equal to $19$, and it was established in \cite{B} that $\Fgn{11}{n}$ is unirational for $n\leq 6$ and uniruled for $n=7$. Theorem \ref{theorem2} leaves $n=8$ and $10$ as the only open cases in $g=11$.

We now give a proof of Theorem \ref{theorem2}.

\begin{proof}[Proof of Theorem \ref{theorem2}]
It was established in \cite[Thm. 0.1]{FV1} that the universal Jacobian of degree $10$ over $\overline{\mathcal{M}}_{10}$ has Kodaira dimension zero. The target of the finite map
$$\Mbar{10}{10}\to \Mbar{10}{10}\big/S_{10}$$
is birational to the universal Jacobian $\sh{J}_{10}^{10}$ and by Theorem \ref{theorem1}, every intermediate quotient has Kodaira dimension zero, in particular 
$${\rm{Kod}}\left(\Mbar{10}{10}\big/S_2\right)=0.$$
Using the birational isomorphism 
$$\Mbar{10}{10}\big/S_2\overset{\sim}{\dashrightarrow}\Fgn{11}{9}$$
constructed in \cite[Thm. 03]{B} we conclude our result. 
\end{proof}

\section{Further applications}\label{FurtherApplications}
\subsection{An extremal nef curve in $\overline{\mathcal{M}}_{10}$}
Consider the curve $B:=\pi_*\Gamma(11,1,0)$ in $\overline{\mathcal{M}}_{10}$ where $\pi:\Mbar{10}{2}\longrightarrow \overline{\mathcal{M}}_{10}$ forgets both marked points. 

\begin{proposition}
$B$ is a nef curve and forms an extremal ray of the cone of nef curves in $ \overline{\mathcal{M}}_{10}$. Further, $\mathbb{Q}^+\langle \overline{\mathcal{K}},\delta_1,\dots \delta_5\rangle$ forms a face of the effective cone of divisors in $ \overline{\mathcal{M}}_{10}$.
\end{proposition}

\begin{proof}
The general curve $B$ is irreducible and as a general genus $g=10$ curve appears as the normalisation of a genus $g=11$ nodal curve on a $K3$ surface, $B$ covers a Zariski dense subset of $\overline{\mathcal{M}}_{10}$ and hence $B$ is nef. Observe $B\cdot\lambda=22$ and $B\cdot \delta_{\rm{irr}}=154$ with intersections equal to zero with all other generators. But then
$$B\cdot \overline{\mathcal{K}}=B\cdot\delta_1=B\cdot\delta_2=B\cdot\delta_3=B\cdot\delta_4=B\cdot\delta_5=0$$
and as all of the above divisors are known to be extremal and the rank of $\Pic(\overline{\mathcal{M}}_g)$ is 7, $\mathbb{Q}^+\langle \overline{\mathcal{K}},\delta_1,\dots \delta_5\rangle$ forms a face of the effective cone. 
\end{proof}

\subsection{Slopes of divisors on $\Mgbar$}
Finally, as corollary of Proposition \ref{intersections} we have the following. See \cite[Prop. 2.2]{FP} for the case $\delta=0$.

\theoremfour*

\begin{proof}
Let $\pi:\Mbar{g}{2\delta}\to\Mgbar$ be the forgetful map and $B=\pi_*\Gamma\left(g+\delta,\delta,0\right)$ the push forward of the curve constructed in Section 2. By Proposition \ref{intersections}
$$\frac{B\cdot\delta}{B\cdot\lambda}=6+\frac{12-\delta}{g+1}>s(D).$$
Therefore $B\cdot D<0$. Since $B$ is irreducible and covers the locus $\mathcal{K}_{g,\delta}\subset\Mg$ of curves that are the normalization of $\delta$-nodal curves lying on a $K3$ surface of genus $g+\delta$, it follows that $\mathcal{K}_{g,\delta}\subset D$.  
\end{proof}


\end{document}